\documentclass{article}
\usepackage{amsmath}
\usepackage{amsfonts}
\usepackage{geometry}

\newtheorem{theorem}{Theorem}

\newtheorem{corollary}[theorem]{Corollary}

\newtheorem{definition}[theorem]{Definition}
\newtheorem{example}[theorem]{Example}

\newtheorem{lemma}[theorem]{Lemma}
\newtheorem{notation}[theorem]{Notation}

\newtheorem{remark}[theorem]{Remark}

\newenvironment{proof}[1][Proof]{\noindent\textbf{#1.} }{\ \rule{0.5em}{0.5em}}
\input{tcilatex}
\begin{document}

\title{An Algorithm for Computing the Ratliff-Rush Closure}
\author{Ibrahim Al-Ayyoub}
\maketitle

\begin{abstract}
Let $I\subset K[x,y]$ be a $\left\langle x,y\right\rangle $-primary monomial
ideal where $K$ is a field. This paper produces an algorithm for computing
the Ratliff-Rush closure $\widetilde{I}$ for the ideal $I=\left\langle
m_{0},\ldots ,m_{n}\right\rangle $ whenever $m_{i}$ is contained in the
integral closure of the ideal $\langle x^{a_{n}},y^{b_{0}}\rangle $. This
generalizes of the work of Crispin \cite{Cri}. Also, it provides
generalizations and answers for some questions given in \cite{HJLS}, and
enables us to construct infinite families of Ratliff-Rush ideals.
\end{abstract}

\bigskip\ \ \ \ \ 

Let $R$ be a commutative Noetherian ring with unity and $I$ a regular ideal
in $R$, that is, an ideal that contains a nonzerodivisor. Then the ideals of
the form $I^{n+1}:I^{n}=\{x\in R\mid xI^{n}\subseteq I^{n+1}\}$ give the
ascending chain $I:I^{0}\subseteq I^{2}:I^{1}\subseteq \ldots \subseteq
I^{n}:I^{n+1}\subseteq \ldots $. Let%
\begin{equation*}
\widetilde{I}=\underset{n\geq 1}{\cup }(I^{n+1}:I^{n}).
\end{equation*}%
As $R$ is Noetherian, $\widetilde{I}$ $=I^{n+1}:I^{n}$ for all sufficiently
large $n$. Ratliff and Rush \cite[Theorem 2.1]{RR}\ proved that $\widetilde{I%
}$\ is the unique largest ideal for which $(\widetilde{I})^{n}=I^{n}$ for
sufficiently large $n$. The ideal $\widetilde{I}$\ is called the \textit{%
Ratliff-Rush\ closure} of $I$ and $I$ is called \textit{Ratliff-Rush} if $I=%
\widetilde{I}$.

\ \ \ \ 

As yet, there is no algorithm to compute the Ratliff-Rush closure for
regular ideals in general. To compute $\cup _{n}(I^{n+1}:I^{n})$ one needs
to find a positive integer $N$ such that $\cup _{n}(I^{n+1}:I^{n})$ $%
=I^{N+1}:I^{N}$. However, $I^{n+1}:I^{n}=I^{n+2}:I^{n+1}$\ does not imply
that $I^{n+1}:I^{n}=I^{n+3}:I^{n+2}$\ (\cite{RS}, Example (1.8)). Several
different approaches have been used to decide the Ratliff-Rush closure;
Heinzer et al. \cite{HLS} , Property (1.2), established that every power of
a regular ideal $I$ is Ratliff-Rush if and only if the associated graded
ring, $gr_{I}(R)=\oplus _{n\geq 0}I^{n}/I^{n+1}$, has a nonzerodivisor (has
positive depth). Thus the Ratliff-Rush property of an ideal is a good tool
for getting information about the depth of the graded associated ring which
is a topic of interest for many authors such as \cite{HM}, \cite{Hun} and 
\cite{Ghe}. Al-Ayyoub \cite{Ayy} used a technique that depends on the degree
count to prove that certain monomial ideals (that are the defining ideal of
certain monomial curves) are Ratliff-Rush, namely, if the ideal $I\subseteq
K[x_{1},\ldots ,x_{n}]$ with $K$ a field\ is primary to $(x_{1},\ldots
,x_{n})$ and $\widetilde{I}\cap (I:(x_{1},\ldots ,x_{n}))\subseteq I$, then $%
I$\ is Ratliff-Rush (for a proof see either Theorem (1.3) in \cite{Ayy} or
Proposition (15.4.1) in \cite{SH}). Elias \cite{Elias} established a
procedure for computing the Ratliff-Rush closure of $\mathbf{m}$-primary
ideals of a Cohen-Macaulay local ring with maximal ideal $\mathbf{m}$.
Elias' procedure depends on computing the Hilbert-Poincar\'{e} series of $I$
and then the multiplicity and the postulation number of $I$.

\ \ \ 

Let $I\subset K[x,y]$ be a $\left\langle x,y\right\rangle $-primary monomial
ideal with $I=\left\langle m_{0},\ldots ,m_{n}\right\rangle $ where $%
m_{i}=x^{a_{i}}y^{b_{i}}$ for $i=0,\ldots ,n$ with $a_{0}=b_{n}=0$. That is, 
$I=\langle y^{b_{0}},x^{a_{1}}y^{b_{1}},\ldots
,x^{a_{n-1}}y^{b_{n-1}},x^{a_{n}}\rangle $. In this paper we produce an
algorithm for computing the Ratliff-Rush closure $\widetilde{I}$ for the
ideal $I$ whenever $m_{i}\in I\left( a_{n},b_{0}\right) $, the integral
closure of the ideal $\langle x^{a_{n}},y^{b_{0}}\rangle $ (see the
definition of the integral closure in the beginning of the next section).
This gives a generalization of the work of Crispin \cite{Cri}. This
algorithm provides generalizations and answers for some questions given in 
\cite{HJLS}. Also, it enables us to construct infinite families of
Ratliff-Rush ideals. We may say that the algorithm we provide in this paper
is the very first explicit algorithm, for computing the Ratliff-Rush closure
for a wide range of monomial ideals in polynomial rings with two
indeterminates, as no theoretical background is needed, that is, the
algorithm depends\ only on elementary computations on numerical semigroups.

\ \ \ \ \ \ \ 

The algorithm is simple enough to be introduced right away and demonstrated
on an example: let $\Omega $ be the numerical semigroup in $\mathbb{Z}^{2}$
generated by the set $\{\left( a_{i},b_{i}\right) \mid i=0,\ldots ,n\}$,
that is, $\Omega =$ $\{(\alpha ,\beta )=\tsum\limits_{i=0}^{n}\lambda
_{i}\left( a_{i},b_{i}\right) \mid \lambda _{i}\in \mathbb{Z}_{\geq 0}\}$.
Let 
\begin{equation*}
S=\{(\alpha ,\beta )\mid \alpha \leq a_{n}\text{, }\beta \leq b_{0},\text{
and }(\alpha ,\beta +kb_{0})\in \Omega \text{ \ for some }k\in \mathbb{Z}_{%
\mathbb{\geq }0}\},
\end{equation*}%
and 
\begin{equation*}
T=\{(\alpha ,\beta )\mid \alpha \leq a_{n},\beta \leq b_{0},\text{ and }%
(\alpha +ka_{n},\beta )\in \Omega \text{ \ for some }k\in \mathbb{Z}_{%
\mathbb{\geq }0}\text{ }\}.
\end{equation*}%
Set 
\begin{equation*}
I_{S}=\left\langle x^{\alpha }y^{\beta }\mid (\alpha ,\beta )\in
S\right\rangle \text{ and }I_{T}=\left\langle x^{\alpha }y^{\beta }\mid
(\alpha ,\beta )\in T\right\rangle \text{.}
\end{equation*}%
\newline
Then we show that $\widetilde{I}=I_{S}\cap I_{T}$.

\ 

Before proceeding to prove this result we would like to demonstrate it by
the example below. The reader may have a look at Example (\ref{ideal-4-gens}%
) which might give an easier representation. A semigroup $S$ in $\mathbb{Z}%
^{2}$ is said to be minimally generated by a set $A\subseteq S$ if $A$ is
the smallest subset in $S$ such that whenever $\left( \alpha ,\beta \right)
\in S$, then there exists $\left( \alpha ^{\prime },\beta ^{\prime }\right)
\in A$ such that $\alpha ^{\prime }\leq \alpha $ and $\beta ^{\prime }\leq
\beta $.

\ 

\begin{example}
Let $I=\left\langle
y^{28},x^{2}y^{26},x^{10}y^{14},x^{11}y^{12},x^{15}y^{5},x^{17}\right\rangle
\subseteq I(17,28)\subset K[x,y]$. Then $\Omega =\langle p_{0},p_{1},p_{2}$, 
$p_{3},p_{4},p_{5}\rangle $ where $p_{0}=(a_{0},b_{0})=(0,28)$, $%
p_{1}=(a_{1},b_{1})=(2,26)$, $p_{2}=(a_{2},b_{2})=(10,14)$, $%
p_{3}=(a_{3},b_{3})=(11,12)$, $p_{4}=(a_{4},b_{4})=(15,5)$, and $%
p_{5}=(a_{5},b_{5})=(17,0)$. To compute $S$ consider $\{(\alpha ,\beta )\in
\Omega $ and $\alpha \leq 17\}=\{\tsum\limits_{i=0}^{5}\lambda _{i}p_{i}\mid
\lambda _{0}\in \mathbb{Z}_{\geq 0}$, $\lambda _{1}\leq 8$, $\lambda
_{i}\leq 1$ for $2\leq i\leq 5\}$. Now $S$ is minimally generated by\ $%
\{p_{0},p_{1},p_{2},p_{3},p_{4},p_{5}\}\cup \{(4,24),(6,22),(8,20),(13,10)\}$
as $(4,24)=(2a_{1},2b_{1}\func{mod}28)$, $(6,22)=(3a_{1},3b_{1}\func{mod}28)$%
, $(8,20)=(4a_{1},4b_{1}\func{mod}28)$, and $(13,10)=(a_{1}+a_{3},b_{1}+b_{3}%
\func{mod}28)$. Thus 
\begin{equation*}
I_{S}=\langle
y^{28},x^{2}y^{26},x^{4}y^{24},x^{6}y^{22},x^{8}y^{20},x^{10}y^{14},x^{11}y^{12},x^{13}y^{10},x^{15}y^{5},x^{17}\rangle .
\end{equation*}%
Similarly, $T$ is minimally generated by $%
\{p_{0},p_{1},p_{2},p_{3},p_{4},p_{5}\}\cup \{(13,10),(9,17),(8,19),(7,22)$, 
$(5,24)\}$ as $(13,10)=(2a_{4}\func{mod}17,2b_{4})$, $(9,17)=(a_{3}+a_{4}%
\func{mod}17,b_{3}+b_{4})$, $(8,19)=(a_{2}+a_{4}\func{mod}17,b_{2}+b_{4})$, $%
(7,22)=(a_{3}+2a_{4}\func{mod}17,b_{3}+2b_{4})$, and $(5,24)=(2a_{3}\func{mod%
}17,2b_{3})$. Thus 
\begin{equation*}
I_{T}=\langle
y^{28},x^{2}y^{26},x^{5}y^{24},x^{7}y^{22},x^{8}y^{19},x^{9}y^{17},x^{10}y^{14},x^{11}y^{12},x^{13}y^{10},x^{15}y^{5},x^{17}\rangle .
\end{equation*}%
Therefore, $\widetilde{I}=I_{S}\cap I_{T}=\langle
y^{28},x^{2}y^{26},x^{5}y^{24},x^{7}y^{22},x^{8}y^{20},x^{10}y^{14},x^{11}y^{12},x^{13}y^{10},x^{15}y^{5},x^{17}\rangle 
$.
\end{example}

The author would like to point out that the algorithm that is provided in
this paper does not apply to arbitrary monomial ideals in $K[x,y]$ as it
will be illustrated at the end of the next section.

\section{\protect\bigskip Decomposition of powers of an ideal}

We start by decomposing sufficiently large powers of the ideal $I$ by means
of the semigroups $S$ and $T$, see Lemma $\left( \ref{MainLemma-One}\right) $
below. In order to do so we need to consider some remarks concerning the
semigroups $S$ and $T$ and the hypothesis $m_{i}\in I\left(
a_{n},b_{0}\right) $, the integral closure of the ideal $\langle
x^{a_{n}},y^{b_{0}}\rangle $ as we define now:

\begin{definition}
Let $I$ be an ideal in a Noetherian ring $R$. The integral closure of $I$ is
the ideal $\overline{I}$ that consists of all elements of $R$ that satisfy
an equation of the form 
\begin{equation*}
x^{n}+a_{1}x^{n-1}+\cdots +a_{n-1}x+a_{n}=0,\ \ \ \ \ a_{i}\in I^{i}.
\end{equation*}%
The ideal $I$ is said to be integrally closed if $I=\overline{I}$.
\end{definition}

\ \ 

It is well known that the integral closure of monomial ideal in a polynomial
ring is again a monomial ideal (See \cite{SH}, Proposition 1.4.2). The
problem of finding the integral closure for a monomial ideal $I$ reduces to
finding monomials $r$, integer $i$ and monomials $m_{1},m_{2},\ldots ,m_{i}$
in $I$ such that $r^{i}=m_{1}m_{2}\cdots m_{i}$, see Section 1.4 in \cite{SH}%
. Geometrically, finding the integral closure of monomial ideals $I$ in $%
R=K[x_{0},\ldots ,x_{n}]$ is the same as finding all the integer lattice
points in the convex hull $NP(I)$\ (the Newton polyhedron of $I$) in $%
\mathbb{R}^{n}$ of $\Gamma (I)$ (the Newton polytope of $I$) where $\Gamma
(I)$ is the set of all exponent vectors of all the monomials in $I$. This
implies $x_{1}^{\gamma _{1}}x_{2}^{\gamma _{2}}\cdots x_{n}^{\gamma _{n}}\in
I(a_{1},a_{2},\ldots ,a_{n})$ if and only if there are non-negative rational
numbers $c_{1},c_{2},\ldots ,c_{n}$ with $\tsum\limits_{i=1}^{n}c_{i}=1$ and 
$\gamma _{i}\geq c_{i}a_{i}$.

\ \ \ \ \ \ \ 

\begin{remark}
\label{I^L} Let $I=\left\langle m_{0},\ldots ,m_{n}\right\rangle $ where $%
m_{i}=x^{a_{i}}y^{b_{i}}\in I\left( a_{n},b_{0}\right) $ for $i=0,\ldots ,n$
with $a_{0}=b_{n}=0$. \ 

(1) $I^{l}=\left\langle x^{\alpha }y^{\beta }:(\alpha ,\beta
)=\tsum\limits_{i=0}^{n}\lambda _{i}\left( a_{i},b_{i}\right) \in \Omega 
\text{ and }\tsum\limits_{i=0}^{n}\lambda _{i}=l\right\rangle $ for all $%
l\in \mathbb{Z}^{+}$.

(2) If $x^{\alpha }y^{\beta }\in I^{l}$, then $\frac{\alpha }{a_{n}}+\frac{%
\beta }{b_{0}}\geq l$. In particular, either $\beta \geq \left( l/2\right)
b_{0}$ or $\alpha \geq \left( l/2\right) a_{n}$.

(3) $\frac{b_{0}-b_{i}}{a_{i}}\leq \frac{b_{0}}{a_{n}}$ for $i=1,\ldots ,n-1$%
.
\end{remark}

\begin{proof}
As $m_{i}=x^{a_{i}}y^{b_{i}}\in I\left( a_{n},b_{0}\right) $, then there
exist $c_{1},c_{2}\in \mathbb{Q}^{\mathbb{+}}$ with $c_{1}+c_{2}=1$ such
that $a_{i}\geq c_{1}a_{n}$ and $b_{i}\geq c_{2}b_{0}$. Hence, $\frac{a_{i}}{%
a_{n}}+\frac{b_{i}}{b_{0}}\geq 1$, this implies $\frac{\alpha ^{\prime }}{%
a_{n}}+\frac{\beta ^{\prime }}{b_{0}}=\tsum\limits_{i=0}^{n}\lambda
_{i}\left( \frac{a_{i}}{a_{n}}+\frac{b_{i}}{b_{0}}\right) \geq
\tsum\limits_{i=0}^{n}\lambda _{i}=l$. Also, $\frac{b_{0}-b_{i}}{a_{i}}\leq 
\frac{b_{0}-b_{i}}{c_{1}a_{n}}\leq \frac{b_{0}-c_{2}b_{0}}{c_{1}a_{n}}=\frac{%
b_{0}}{a_{n}}\frac{1-c_{2}}{c_{1}}=\frac{b_{0}}{a_{n}}$.
\end{proof}

\ \ \ \ \ \ \ \ \ \ \ \ \ \ 

The following remark provides us with the technique that we repeatedly use
in this paper

\begin{remark}
\label{L-1-equality}If $(\alpha ,\beta )=\tsum\limits_{i=0}^{n}\lambda
_{i}\left( a_{i},b_{i}\right) \in \Omega $ with $\tsum\limits_{i=0}^{n}%
\lambda _{i}=l$ and $\alpha =\tsum\limits_{i=0}^{n}\lambda _{i}a_{i}\leq
a_{n}$, then $\beta =\tsum\limits_{i=0}^{n}\lambda _{i}b_{i}=\left(
l-1\right) b_{0}+\beta _{1}$ with $\left( \alpha ,\beta _{1}\right) \in S$.
Also, if $(\alpha ,\beta )\in S$ with $(\alpha ,\beta +kb_{0})\in \Omega $
and $\alpha =\tsum\limits_{i=0}^{n}\lambda _{i}a_{i}$, $\beta
+kb_{0}=\tsum\limits_{i=0}^{n}\lambda _{i}b_{i}$ with $\tsum%
\limits_{i=0}^{n}\lambda _{i}=l$, then $k=l-1$.
\end{remark}

\begin{proof}
Showing $\beta =\tsum\limits_{i=0}^{n}\lambda _{i}b_{i}=\left( l-1\right)
b_{0}+\beta _{1}$ with $\beta _{1}\leq b_{0}$ is equivalent to showing that $%
\tsum\limits_{i=0}^{n}\lambda _{i}\left( b_{0}-b_{i}\right) \leq b_{0}$.
Since $\alpha =\tsum\limits_{i=0}^{n}\lambda _{i}a_{i}\leq a_{n}$, then $%
\tsum\limits_{i=0}^{n}\lambda _{i}\frac{a_{i}}{a_{n}}\leq 1$. Thus by part
(3) of Remark $\left( \ref{I^L}\right) $ we get $\tsum\limits_{i=0}^{n}%
\lambda _{i}\left( b_{0}-b_{i}\right) \leq \tsum\limits_{i=0}^{n}\lambda _{i}%
\frac{a_{i}}{a_{n}}b_{0}\leq b_{0}$.

\ \ 

To prove the other part it is enough to show $\beta +kb_{0}\geq \left(
l-1\right) b_{0}$. Consider $\left( l-1\right) b_{0}-\left( \beta
+kb_{0}\right) =lb_{0}-\tsum\limits_{i=0}^{n}\lambda
_{i}b_{i}-b_{0}=\tsum\limits_{i=0}^{n}\lambda _{i}\left( b_{0}-b_{i}\right)
-b_{0}\leq 0$ since $\tsum\limits_{i=0}^{n}\lambda _{i}\left(
b_{0}-b_{i}\right) \leq b_{0}$ as above.
\end{proof}

\ \ \ 

\begin{notation}
\label{Notations} Let $q_{_{S}},q_{_{T}}\in \mathbb{Z}$ be such that $%
a_{n}=q_{_{S}}a_{1}+e_{_{S}}$ with $0\leq e_{_{S}}<a_{1}$ and $%
b_{n}=q_{_{T}}b_{n-1}+e_{_{T}}$ with $0\leq e_{_{T}}<b_{n-1}$. Note that if $%
(\alpha ,\beta )\in S$ with $\alpha =\tsum\limits_{i=0}^{n}\lambda _{i}a_{i}$%
, then $\tsum\limits_{i=0}^{n}\lambda _{i}\leq q_{_{S}}$. And if \ $(\alpha
,\beta )\in T$ with $\beta =\tsum\limits_{i=0}^{n}\delta _{i}b_{i}$, then $%
\tsum\limits_{i=0}^{n}\delta _{i}\leq q_{_{T}}$. Also, note $%
q_{_{S}},q_{_{T}}\geq 1$.
\end{notation}

\ \ \ \ \ \ \ \ \ \ \ \ \ \ \ \ 

This section is concluded with an explicit decomposition of sufficiently
large powers of the ideal $I$. This decomposition enables us to compute the
Ratliff-Rush closure.

\begin{lemma}
\label{MainLemma-One}Let $I$, $I_{S}$, and $I_{T}$ be as above. Then for
every $l\geq \max \{q_{_{T}},q_{_{S}}\}$%
\begin{equation*}
I^{l}=y^{b_{0}(l-1)}I_{S}+x^{a_{n}(l-1)}I_{T}+x^{a_{n}}y^{b_{0}}M
\end{equation*}%
where $M=I^{l}:(x^{a_{n}}y^{b_{0}})$.
\end{lemma}

\begin{proof}
If $x^{\alpha }y^{\beta }\in I_{T}$, then $\beta
=\tsum\limits_{i=0}^{n}\lambda _{i}b_{i}\leq b_{0}$ and $\alpha
=\tsum\limits_{i=0}^{n}\lambda _{i}a_{i}-ca_{n}$ for some positive integer $%
c $ with $q_{_{T}}\geq \tsum\limits_{i=0}^{n}\lambda _{i}>c$. Now if $l\geq
q_{_{T}}$, then $x^{a_{n}(l-1)}x^{\alpha }y^{\beta }=\left( x^{a_{n}}\right)
^{l-(c+1)}\tprod\limits_{i=0}^{n}\left( x^{a_{i}}y^{b_{i}}\right) ^{\lambda
_{i}}\in I^{l}$ as $l-(c+1)+\tsum\limits_{i=0}^{n}\lambda _{i}\geq l$.
Similarly, if $l\geq q_{_{S}}$, then $y^{b_{0}(l-1)}x^{\alpha }y^{\beta }\in
I^{l}$ for every $x^{\alpha }y^{\beta }\in I_{S}$.

\ \ \ \ \ \ 

For the other inclusion it is enough to show that if $x^{\alpha }y^{\gamma
}\in I^{l}$ with $\alpha \leq a_{n}$, then $x^{\alpha }y^{\gamma }\in
y^{b_{0}(l-1)}I_{S}$. But this is done by part (1) of Remark $\left( \ref%
{I^L}\right) $ and Remark $\left( \ref{L-1-equality}\right) $ as $\alpha
=\tsum\limits_{i=0}^{n}\lambda _{i}a_{i}\leq a_{n}$ and $\gamma
=\tsum\limits_{i=0}^{n}\lambda _{i}b_{i}=(l-1)b_{0}+\beta $ with $(\alpha
,\beta )\in I_{S}$.
\end{proof}

\ \ \ 

Considering the ideal $I=\left\langle
x^{7},x^{6}y,xy^{10},y^{14}\right\rangle $, the reader can easily see that
any power of $I$ does not satisfy the above decomposition which is a
cornerstone of the main result of this paper. This causes the algorithm not
to be applicable for arbitrary monomial ideals.

\section{Powers of an ideal and the Ratliff-Rush closure}

In the lemma below we show that the generators of a sufficiently large power
of $I$ take a patterns that involve powers of $m_{0}$ and $m_{n}$. This is a
consequences of the hypothesis on the generators of $I$, that is, $%
m_{i}=x^{a_{i}}y^{b_{i}}\in I\left( a_{n},b_{0}\right) $.

\begin{lemma}
\label{MainLemma-Two}Let $I=\left\langle m_{n},\ldots ,m_{0}\right\rangle $
where $m_{i}=x^{a_{i}}y^{b_{i}}\in I\left( a_{n},b_{0}\right) $ for $%
i=0,\ldots ,n$ and $a_{0}=b_{n}=0$. Let $r$ be any positive integer. Then
there exist a positive integer $L$ such that if $l\geq L$, then the
generators of $I^{l}$ are of the forms $m_{0}^{\gamma }\xi _{0,\gamma
}m_{n}^{l-\gamma -1}$ and $m_{0}^{l-\gamma -1}\xi _{\gamma ,0}m_{n}^{\gamma
} $ for every $\gamma $ with $r\leq \gamma \leq l-1$ where $\xi _{0,\gamma }$
and $\xi _{\gamma ,0}$ are some monomials.
\end{lemma}

\begin{proof}
Let $r$ be a positive integer and $q=\max \{q_{_{T}},q_{_{S}}\}$. Choose $%
L=2\left( r+1\right) $ and let $\omega =x^{\alpha ^{\prime }}y^{\beta
^{\prime }}\in I^{l}$ for some $l$ $\geq L$. By part (1) of Remark $\left( %
\ref{I^L}\right) $ we may write $(\alpha ^{\prime },\beta ^{\prime
})=\tsum\limits_{i=0}^{n}\lambda _{i}\left( a_{i},b_{i}\right) \in \Omega $
with $\tsum\limits_{i=0}^{n}\lambda _{i}=l$. By part (2) of Remark $\left( %
\ref{I^L}\right) $ either $\beta ^{\prime }\geq \left( r+1\right) b_{0}$ or $%
\alpha ^{\prime }\geq \left( r+1\right) a_{n}$. We make the proof whenever $%
\beta ^{\prime }\geq \left( r+1\right) b_{0}$ where we show $\omega
=m_{0}^{\gamma }\xi _{0,\gamma }m_{n}^{l-\gamma -1}$ for some monomial $\xi
_{0,\gamma }$. The proof is similar for the case $\alpha ^{\prime }\geq
\left( r+1\right) a_{n}$ where it can be shown that $\omega =m_{0}^{l-\gamma
-1}\xi _{\gamma ,0}m_{n}^{\gamma }$.

\ \ \ \ \ 

Let $\beta ^{\prime }\geq \left( r+1\right) b_{0}$ and write $\beta ^{\prime
}=\gamma b_{0}+\beta $ with $0\leq \beta <b_{0}$. Note $r<\gamma <l$ as $%
\left( r+1\right) b_{0}\leq \beta ^{\prime }\leq lb_{0}$. Since $\frac{\beta
^{\prime }}{b_{0}}<\gamma +1$, then by part (2) of Remark $\left( \ref{I^L}%
\right) $ we must have $\alpha ^{\prime }\geq (l-\gamma -1)a_{n}$. Write $%
\alpha ^{\prime }=(l-\gamma -1)a_{n}+\alpha $. Now $\omega =x^{\alpha
^{\prime }}y^{\beta ^{\prime }}=(y^{b_{0}})^{\gamma }x^{\alpha }y^{\beta
}\left( x^{a_{n}}\right) ^{l-\gamma -1}=m_{0}^{\gamma }x^{\alpha }y^{\beta
}m_{n}^{l-\gamma -1}$.

\ \ \ \ 

Finally, note $\left\langle x^{a_{n}},y^{b_{0}}\right\rangle \subseteq I^{l}$%
, hence $\left\langle y^{lb_{0}},y^{(l-1)b_{0}}x^{a_{n}},\ldots
,y^{(r+1)b_{0}}x^{l-(r+1)a_{n}},\ldots ,x^{a_{n}}\right\rangle \subseteq
I^{l}$ which suffices to show that $\gamma $ takes all integer values
between $r$ and $l-1$, which finishes the proof.
\end{proof}

\begin{remark}
\label{m-0^q}If $r=2q$ and $l\geq 4q+2$ as in the above lemma and $\omega
\in I^{l}$, then $\omega \in (m_{0}^{q})I^{l-q}$ or $\omega \in
I^{l-q}(m_{n}^{q})$.
\end{remark}

\begin{proof}
Assume $\omega =m_{0}^{\gamma }\xi _{0,\gamma }m_{n}^{l-\gamma -1}\in I^{l}$
with $r\leq \gamma \leq l-1$. Applying the above lemma with $r^{\prime }=q-1$
and $l^{\prime }\geq 2q$, then the generators of $I^{l^{\prime }}$ are of
the forms $m_{0}^{\gamma ^{\prime }}\xi _{0,\gamma ^{\prime
}}m_{n}^{l-\gamma ^{\prime }-1}$ and $m_{0}^{l-\gamma ^{\prime }-1}\xi
_{\gamma ^{\prime },0}m_{n}^{\gamma ^{\prime }}$ for every $\gamma ^{\prime
} $ with $q-1\leq \gamma ^{\prime }\leq l-1$ where $\xi _{0,\gamma ^{\prime
}}$ and $\xi _{\gamma ^{\prime },0}$ are some monomials. As\ $2q\leq \gamma $
and $l\geq 4q+2$, then $r^{\prime }\leq \gamma -q-1\leq l-q-2$ and $%
l-q-2\geq 2q$. Thus setting $\gamma ^{\prime }=\gamma -q-1$ and $l^{\prime
}=l-q-2$. Therefore, $m_{0}^{\gamma -q-1}\xi _{0,\gamma ^{\prime
}}m_{n}^{l-\gamma -2}=m_{0}^{\gamma ^{\prime }}\xi _{0,\gamma ^{\prime
}}m_{n}^{l-\gamma ^{\prime }-1}\in I^{l-q}$. Note $m_{0}\xi _{0,\gamma
}m_{n} $ is a multiple of $\xi _{0,\gamma ^{\prime }}$, say $m_{0}\xi
_{0,\gamma }m_{n}=\rho \xi _{0,\gamma ^{\prime }}$. Thus $\omega
=m_{0}^{\gamma }\xi _{0,\gamma }m_{n}^{l-\gamma -1}=$ $m_{0}^{\gamma
-1}\left( m_{0}\xi _{0,\gamma }m_{n}\right) m_{n}^{l-\gamma -2}=$ $%
m_{0}^{q}\rho (m_{0}^{\gamma -q-1}\xi _{0,\gamma ^{\prime }}m_{n}^{l-\gamma
-2})\in (m_{0}^{q})I^{l-q}$.

\ \ \ \ 

Assume $\omega =m_{0}^{l-\gamma -1}\xi _{\gamma ,0}m_{n}^{\gamma }$. Then a
similar process shows that $\omega \in I^{l-q}(m_{n}^{q})$, which finishes
the proof.
\end{proof}

\ \ \ \ 

Now we are ready to prove the first main theorem of the paper.

\begin{theorem}
\label{RR-Closure}Let the ideals $I$, $I_{S}$, and $I_{T}$ be as before.
Then $\widetilde{I}=I_{S}\cap I_{T}$.
\end{theorem}

\begin{proof}
Let $\delta \in I_{S}\cap I_{T}$ and $q=\max \{q_{_{S}},q_{_{T}}\}$. Claim $%
\delta m_{0}^{q},\delta m_{n}^{q}\in I^{q+1}$: as $\delta \in I_{S}$, then $%
\delta =$ $x^{r}y^{s}$ with $\left( r-u,s-v\right) \in S$ for some positive
integers $u$ and $v$, that is, $r-u=\tsum\limits_{i=0}^{n}\lambda
_{i}a_{i}\leq a_{n}$ and $s-v\leq b_{0}$ with $(r-u,s-v+kb_{0})\in \Omega $
and $s-v+kb_{0}=\tsum\limits_{i=0}^{n}\lambda _{i}b_{i}$. Let $%
t=\tsum\limits_{i=0}^{n}\lambda _{i}$. By Notations $\left( \ref{Notations}%
\right) $ we have $t\leq q$, and by Remark $\left( \ref{L-1-equality}\right) 
$\ we may rewrite $\tsum\limits_{i=0}^{n}\lambda _{i}b_{i}=(t-1)b_{0}+\ s-v$%
. Then 
\begin{eqnarray*}
\delta m_{0}^{q} &=&x^{r}y^{s}\left( y^{b_{0}}\right)
^{q}=x^{u}y^{v}x^{r-u}y^{s-v}\left( y^{b_{0}}\right) ^{t-1}\left(
y^{b_{0}}\right) ^{q-(t-1)} \\
&=&x^{u}y^{v}x^{\left( \tsum\limits_{i=0}^{n}\lambda _{i}a_{i}\right)
}y^{\left( \tsum\limits_{i=0}^{n}\lambda _{i}b_{i}\right) }\left(
y^{b_{0}}\right) ^{q-(t-1)} \\
&=&x^{u}y^{v}\tprod\limits_{i=0}^{n}\left( x^{a_{i}}y^{b_{i}}\right)
^{\lambda _{i}}\left( y^{b_{0}}\right) ^{q-(t-1)}\in I^{q+1}
\end{eqnarray*}%
Similarly, as $\delta \in I_{T}$, then by a similar procedure as above it
can be shown that $\delta m_{n}^{q}\in I^{q+1}$.

\ \ \ \ \ \ \ \ 

Now choose $r=2q$, then by Lemma $\left( \ref{MainLemma-Two}\right) $ if $%
l\geq 2(2q+1)$, then any generator $\omega $ of $I^{l}$ is either of the
form $m_{0}^{\gamma }\xi _{0,\gamma }m_{n}^{l-\gamma -1}$ or $%
m_{0}^{l-\gamma -1}\xi _{\gamma ,0}m_{n}^{\gamma }$ for every $\gamma $ with 
$r\leq \gamma \leq l-1$ where $\xi _{0,\gamma }$ and $\xi _{\gamma ,0}$ are
some monomials.

\ \ \ \ 

Assume $\omega =m_{0}^{\gamma }\xi _{0,\gamma }m_{n}^{l-\gamma -1}$. By
Remark (\ref{m-0^q}) we have $\omega \in (m_{0}^{q})I^{l-q}$. Therefore, by
the claim above $\delta \omega \in (\delta m_{0}^{q})I^{l-q}\in
I^{q+1}I^{l-q}=I^{l+1}$.

\ \ 

Assume $\omega =m_{0}^{l-\gamma -1}\xi _{\gamma ,0}m_{n}^{\gamma }$. By
Remark (\ref{m-0^q}) we have $\omega \in I^{l-q}(m_{n}^{q})$. Therefore, by
the claim above $\delta \omega \in I^{l-q}(\delta m_{n}^{q})\in
I^{l-q}I^{q+1}=I^{l+1}$.

\ \ \ \ 

On the other hand, assume $\delta \notin I_{S}$ and let $l$ be any positive
integer. Then $\delta y^{lb_{0}}\notin y^{lb_{0}}I_{S}$, also $\delta
y^{lb_{0}}\notin x^{la_{n}}I_{T}$ and $\delta y^{lb_{0}}\notin \left(
x^{a_{n}}y^{b_{0}}\right) M$ because of the $y$-degree count where $%
M=I^{l}:(x^{a_{n}}y^{b_{0}})$. \ Hence, $\delta y^{lb_{0}}\notin I^{l+1}$ by
Lemma $\left( \ref{MainLemma-One}\right) $. Analogously, if $\delta \notin
I_{T}$, then $\delta x^{la_{n}}\notin I^{l+1}$, which finishes the proof.
\end{proof}

\begin{remark}
The Ratliff-Rush reduction number of an ideal $I$ is defined $r(I)=\min
\{l\in \mathbb{Z}_{\geq 0}\mid \widetilde{I}=(I^{l+1}:I^{l})\}$. From the
proof of Theorem $\left( \ref{RR-Closure}\right) $ it is clear that $2q$ is
an upper bound for the Ratliff-Rush reduction number of the ideal $I$.
\end{remark}

\section{Consequences and Examples}

Heinzer et al. \cite{HJLS}, Example (6.3), conjectured that for any integer $%
d$ the ideal $I_{d}=\langle x^{d},x^{d-1}y,y^{d}\rangle $ and all its powers
are Ratliff-Rush. This conjectured was proved later by \cite{RS},
Proposition (1.9), by actual computations of the depth of $gr_{_{I_{d}}}$,
the associated graded ring of $I_{d}$. Later \cite{Cri}, Example (4.2),
proved this conjecture by a method that we generalize in the paper. In
Corollary (\ref{c-m-d}) below we give a generalization of this conjecture.

\begin{remark}
\label{a_i>a/2}Let $I=\langle y^{b_{0}},x^{a_{1}}y^{b_{1}},\ldots
,x^{a_{n-1}}y^{b_{n-1}},x^{a_{n}}\rangle $ with $m_{i}\in I\left(
a_{n},b_{0}\right) $. Then $I$ is Ratliff-Rush if any of the following holds.

\ \ 

(1) $a_{i}\geq a_{n}/2$ for all $i$ or $b_{i}\geq b_{0}/2$ for all $i$.

(2) For all $i$ and $j$ either $a_{i}+a_{j}\geq a_{n}$ or $a_{i}+a_{j}=a_{k}$
and $\left( b_{i}+b_{j}\right) \func{mod}b_{0}\geq b_{k}$ for some $k$.
\end{remark}

Powers of a Ratliff-Rush ideal need not be Ratliff-Rush as Example $\left(
6.1\right) $ of [HJLS] shows. As the powers of an ideal are Ratliff-Rush
implies that the associated graded ring, $gr_{I}(R)=\oplus _{n\geq
0}I^{n}/I^{n+1}$, has a positive depth, we will investigate the Ratliff-Rush
closedness for all powers of ideals in the remaining of the paper.

\begin{corollary}
\label{c-m-d}Let $I=\langle x^{c},m,y^{d}\rangle $ where $m=0$ or $%
m=x^{u}y^{v}$ $\in I\left( c,d\right) $. Then all powers of $I$ are
Ratliff-Rush.
\end{corollary}

\begin{proof}
if $m=0$, then $I=\langle x^{c},y^{d}\rangle $. Thus $S=T=\{(0,c),(d,0)\}$,
hence $I_{S}=I=I_{T}$. Also, $I^{l}=\langle
y^{ld},x^{c}y^{(l-1)d},x^{2c}y^{(l-2)d},\ldots
,x^{(l-2)c}y^{2d},x^{(l-1)c}y^{d},x^{lc}\rangle $. It is clear that $%
(I^{l})_{S}=I^{l}$ and $(I^{l})_{T}=I^{l}$. Thus $\widetilde{I^{l}}=I^{l}$.

\ \ \ 

Assume $m=x^{u}y^{v}$, then by part (2) of Remark ($\ref{I^L}$) either $%
u\geq c/2$ or $v\geq d/2$. Thus $\widetilde{I}=I$ by part (1) of Remark (\ref%
{a_i>a/2}) . Let $J=I^{l}$. Consider 
\begin{eqnarray*}
J &=&\langle y^{ld},x^{u}y^{(l-1)d+v},x^{2u}y^{(l-2)d+2v},\ldots
,x^{lu}y^{lv}\rangle + \\
&&\langle x^{lu}y^{lv},x^{(l-1)u+c}y^{(l-1)v},x^{(l-2)u+2c}y^{(l-2)v},\ldots
,x^{2u+(l-2)c}y^{2v},\ x^{u+(l-1)c}y^{v},x^{lc}\rangle \\
&=&\langle x^{iu}y^{(l-i)d+iv}\mid i=0,\ldots ,l\rangle +\langle
x^{iu+(l-i)c}y^{iv}\mid i=0,\ldots ,l\rangle \text{.}
\end{eqnarray*}%
Let%
\begin{equation*}
K=\left\{ \left( a_{i},b_{i}\right) =\left( iu,(l-i)d+iv\right) \mid
i=0,\ldots ,l\right\}
\end{equation*}%
and%
\begin{equation*}
H=\left\{ \left( a_{2l-i+1},b_{2l-i+1}\right) =\left( iu+(l-i)c,iv\right)
\mid i=0,\ldots ,l\right\} .
\end{equation*}%
Then $J=\left\langle x^{a}y^{b}\mid \left( a,b\right) \in K\cup
H\right\rangle =\left\langle x^{a_{i}}y^{b_{i}}\mid i=0,\ldots
,2l+1\right\rangle $. See Figure (1) for a representation of $J$. Note that
if $\left( a_{i},b_{i}\right) \in K$, then $b_{i}\geq lv$, and if $\left(
a_{i},b_{i}\right) \in H$, then $a_{i}\geq lu$.

\ \ \ 

\underline{Claim}: if $\left( \alpha ,\beta \right) \in S\backslash \left(
K\cup H\right) $, then $\alpha \geq lu$, and if $\left( \alpha ^{\prime
},\beta ^{\prime }\right) \in T\backslash \left( K\cup H\right) $, then $%
\beta ^{\prime }\geq lv$. We prove the first part of the claim and the
second part is similar. Assume $(\alpha ,\beta )\in S$ with $\alpha <lu$. As 
$\alpha =\tsum\limits_{i=0}^{2l+1}\lambda _{i}a_{i}<lu$, then $\lambda
_{i}=0 $ whenever $\left( a_{i},b_{i}\right) \in H$. Hence we must have $%
\alpha =\tsum\limits_{i=0}^{l}\lambda _{i}a_{i}$, that is $\left(
a_{i},b_{i}\right) \in K$, or $\left( a_{i},b_{i}\right) =\left(
iu,(l-i)d+iv\right) $. Thus $\alpha =\tsum\limits_{i=0}^{l}\left( i\lambda
_{i}\right) u$\ and%
\begin{eqnarray*}
\beta &=&\left( \tsum\limits_{i=0}^{l}\lambda _{i}\left( l-i\right)
d+\tsum\limits_{i=0}^{l}\left( i\lambda _{i}\right) v\right) \func{mod}ld \\
&=&\left( \tsum\limits_{i=0}^{l}(\lambda
_{i}-1)ld+\tsum\limits_{i=0}^{l}[(l\ -i\lambda _{i})d+\left( i\lambda
_{i}\right) v]\right) \func{mod}ld \\
&=&\left( \tsum\limits_{i=0}^{l-1}ld+(l\ -\tsum\limits_{i=0}^{l}i\lambda
_{i})d+\left( \tsum\limits_{i=0}^{l}i\lambda _{i}\right) v\right) \func{mod}%
ld \\
&=&(l\ -\tsum\limits_{i=0}^{l}i\lambda _{i})d+\left(
\tsum\limits_{i=0}^{l}i\lambda _{i}\right) v\text{.}
\end{eqnarray*}%
Now as $\alpha =\tsum\limits_{i=0}^{l}\left( i\lambda _{i}\right) u<lu$,
then $\tsum\limits_{i=0}^{l}\left( i\lambda _{i}\right) <l$ and hence $%
(\alpha ,\beta )\in K$. \ \ 

\ \ \ \ \ \ 

Now by the claim, if $x^{\alpha }y^{\beta }\in J_{S}\cap J_{T}$, then either 
$x^{\alpha }y^{\beta }=x^{a_{i}}y^{b_{i}}\in J$ for some $i=0,\ldots ,2l+1$,
or $x^{\alpha }y^{\beta }\in \langle x^{lu}y^{lv}\rangle \subseteq J$.
\end{proof}

\ \ \ \ \ \ \ 

\begin{center}
$\ 
\begin{tabular}{l}
\ \ \ \ \ \ \FRAME{itbpF}{4.5238in}{2.4621in}{0in}{}{}{Figure}{\special%
{language "Scientific Word";type "GRAPHIC";display "USEDEF";valid_file
"T";width 4.5238in;height 2.4621in;depth 0in;original-width
4.8032in;original-height 2.6282in;cropleft "0";croptop "1";cropright
"1";cropbottom "0";tempfilename 'L8738Y06.wmf';tempfile-properties "XPR";}}
\\ 
$\text{\textbf{Figure 1.} }${\small A representation of }${\small I}^{5}$%
{\small \ where }${\small I=}\langle {\small x}^{7}{\small ,x}^{5}{\small y}%
^{2}{\small ,y}^{5}\rangle ${\small . The circles (black or white) } \\ 
{\small represent the set }${\small K}${\small , and the black points
(circles or squares) represent the set }${\small H}${\small . The } \\ 
{\small white square represents a monomial in }${\small I}_{T}${\small \ but
not in }${\small I}${\small .}%
\end{tabular}%
$
\end{center}

\ \ \ 

The above corollary showed that all powers of a $(x,y)$-primary monomial
ideal with three generators, satisfying the underlined conditions, are
Ratliff-Rush. This is not the case if the ideal is generated by $4$ elements
as the example below shows.\ \ 

\ \ \ \ 

\begin{example}
\ \label{ideal-4-gens}Let $I=\left\langle
x^{35},x^{33}y^{2},x^{4}y^{26},y^{28}\right\rangle $. Then $S=\{(0,28),$ $%
(4,26),$ $(8,24),$ $(12,22),$ $(16,20),$ $(20,18),(24,16),$ $%
(28,14),(32,12),(33,2),$ $(35,0)\}$ and $T=\{(35,0),$ $(33,2),$ $(31,4),$ $%
(29,6),$ $(27,8),$ $(25,10),$ $(23,12),$ $(21,14),$ $(19,16),$ $(17,18),$ $%
(15,20),$ $(13,22),$ $(11,24),$ $(4,26),$ $(0,28)\}$ (see the figure below
for illustration). Thus $\widetilde{{\small I}}=\langle {\small x}^{35}%
{\small ,}$ ${\small x}^{33}{\small y}^{2}{\small ,}$ ${\small x}^{32}%
{\small y}^{12}{\small ,}$ ${\small x}^{28}{\small y}^{14},$ ${\small x}^{24}%
{\small y}^{16}{\small ,}$ ${\small x}^{20}{\small y}^{18}{\small ,}$ $%
{\small x}^{16}{\small y}^{20}{\small ,}$ ${\small x}^{13}{\small y}^{22}%
{\small ,}$ ${\small x}^{11}{\small y}^{24}{\small ,}$ ${\small x}^{4}%
{\small y}^{26}{\small ,}$ ${\small y}^{28}\rangle $.
\end{example}

\ \ \ \ \ \ \ \ \ 

\begin{center}
$%
\begin{tabular}{l}
$\ \ \ \ \ \ \ \ \ \ \ \ \ \FRAME{itbpF}{4.0646in}{2.3834in}{0in}{}{}{Figure%
}{\special{language "Scientific Word";type "GRAPHIC";display
"USEDEF";valid_file "T";width 4.0646in;height 2.3834in;depth
0in;original-width 4.9173in;original-height 3.2759in;cropleft "0";croptop
"1";cropright "1";cropbottom "0";tempfilename
'L8739107.wmf';tempfile-properties "XPR";}}$ \\ 
$\text{\textbf{Figure 2.} }${\small A representation of }${\small I}_{%
{\small S}}\backslash I${\small \ (white circles) and }${\small I}_{{\small T%
}}\backslash I$ {\small (black squares) where} \\ 
{\small \ }${\small I=}\langle {\small x}^{35}{\small ,x}^{33}{\small y}^{2}%
{\small ,x}^{4}{\small y}^{26}{\small ,y}^{28}\rangle ${\small . The points
with a slash mark represent the monomials in }$\widetilde{{\small I}}%
\backslash I$.%
\end{tabular}%
$
\end{center}

\ \ \ \ \ \ \ \ 

Crispin \cite{Cri}\ showed that for any $d$ and $k$ the ideal $%
I_{d,k}=\langle y^{d},x^{d-k}y^{k},x^{d-k+1}y^{k-1},\ldots ,$ $%
x^{d-1}y,x^{d}\rangle $ and all its powers are Ratliff-Rush. Also in Example
(4.4) she showed that the ideal $I_{m,k}=\left\langle
x^{im}y^{m(k+1-i)-1}\right\rangle _{i=0}^{k}+\left\langle
x^{km+j}y^{m-j-1}\right\rangle _{j=0}^{m-1}$ and all its powers are
Ratliff-Rush. In corollary (\ref{ideal-Sigma}) below we generalize this.
First consider the notations below and the figures for illustration of the
hypothesis of the corollary.

\ \ 

\begin{notation}
Let $c\leq d$ be two integers and $\mu _{i}=\left\lceil (c-i)\frac{d}{c}%
\right\rceil $. Let $c=n_{1}c_{1}+n_{2}c_{2}+\ldots +n_{r}c_{r}$ with $%
c_{i+1}\ $divides $c_{i}$ and $n_{i}\in \mathbb{Z}^{+}$ for all $i$ and let $%
n_{0}=1$ and $c_{0}=0$. Also let $\sigma
_{j,q}=qc_{j+1}+\tsum\limits_{i=1}^{j}n_{i}c_{i}$. Note the following

\ \ 

(1) $\sigma _{-1,1}=0$, hence $\mu _{\sigma _{-1,1}}=d$.

(2) $\sigma _{0,q}=qc_{1}$ for $q=1,\ldots ,n_{1}$.

(3) $\sigma _{r-1,n_{r}}=c$.

\ \ \newline
Define the ideal $I=\left\langle x^{\sigma _{j,q}}y^{\mu _{\sigma
_{j,q}}}\mid j=-1,0,\ldots ,r-1\text{ and }q=1,2,\ldots
,n_{j+1}\right\rangle $.
\end{notation}

\ \ \ \ 

Note that if $c=d$ and if we choose $r=2,c_{1}=k,n_{1}=1,c_{2}=1$, and $%
n_{2}=d-k$, then we get the ideal $I=I_{d,k}$ mentioned above. Also, if \ $m$
and $k$ are integers and if $c=d=m(k+1)-1$ and if we choose $%
r=2,c_{1}=m,n_{1}=k,c_{2}=1$, and $n_{2}=m-1$, then we get the ideal $%
I=I_{m,k}$ as above. \ \ 

\ \ \ \ \ \ 

$\ \ 
\begin{tabular}{ll}
\FRAME{itbpF}{2.5547in}{2.6697in}{0in}{}{}{Figure}{\special{language
"Scientific Word";type "GRAPHIC";display "USEDEF";valid_file "T";width
2.5547in;height 2.6697in;depth 0in;original-width 3.0113in;original-height
3.48in;cropleft "0";croptop "1";cropright "1";cropbottom "0";tempfilename
'L8739408.wmf';tempfile-properties "XP";}} & \FRAME{itbpF}{2.3151in}{2.7051in%
}{0.0726in}{}{}{Figure}{\special{language "Scientific Word";type
"GRAPHIC";display "USEDEF";valid_file "T";width 2.3151in;height
2.7051in;depth 0.0726in;original-width 2.9179in;original-height
3.2327in;cropleft "0";croptop "1";cropright "1";cropbottom "0";tempfilename
'L8739509.wmf';tempfile-properties "XP";}} \\ 
\textbf{Figure 3.} ${\small d=20,c=17,r=3,c}_{1}{\small =4,}$ & \textbf{%
Figure 4.} ${\small d=20,c=17,r=2,}$ \\ 
${\small n}_{1}{\small =2,c}_{2}{\small =2,n}_{2}{\small =3,c}_{3}{\small %
=1,n}_{3}{\small =3}.$ & ${\small c}_{1}{\small =5,n}_{1}{\small =1,c}_{2}%
{\small =1,n}_{2}{\small =12.}$%
\end{tabular}%
$

\ \ \ \ \ \ \ 

\begin{corollary}
\label{ideal-Sigma} All powers of the ideal $I=\left\langle x^{\sigma
_{j,q}}y^{\mu _{\sigma _{j,q}}}\mid j=-1,0,\ldots ,r-1\text{ and }%
q=1,2,\ldots ,n_{j+1}\right\rangle $ are Ratliff-Rush.
\end{corollary}

\begin{proof}
First note that since $c_{i+1}\ $divides $c_{i}$, then $\sigma _{j,q}$ can
be written as $tc_{j+1}$ for some $t<n_{j+1}$, or as $%
n_{j+1}c_{j+1}+n_{j+2}c_{j+2}+\ldots +n_{j+e}c_{j+e}+tc_{j+e+1}$ for some $e$
and $t<n_{j+e+1}$. Thus if $\sigma _{j_{1},q_{1}}+\sigma _{j_{2},q_{2}}\leq
c $, then $\sigma _{j_{1},q_{1}}+\sigma _{j_{2},q_{2}}=\sigma _{j,q}$ for
some $j\geq \max \{j_{1},j_{2}\}$. Also $2d>\mu _{\sigma _{j_{1},q_{1}}}+\mu
_{\sigma _{j_{2},q_{2}}}\geq d+\left\lceil (c-\sigma _{j,q})\frac{d}{c}%
\right\rceil $, hence $(\mu _{\sigma _{j_{1},q_{1}}}+\mu _{\sigma
_{j_{2},q_{2}}})\func{mod}d\geq \left\lceil (c-\sigma _{j,q})\frac{d}{c}%
\right\rceil =\mu _{\sigma _{j,q}}$. Thus $I$ is Ratliff-Rush by part (2) of
Remark (\ref{a_i>a/2}).

\ \ 

Let $\omega =x^{a}y^{b}$ and $\omega ^{\prime }=x^{a^{\prime }}y^{b^{\prime
}}\ $be generators of\ $I^{l}$ for some $l$. If $a+a^{\prime }\leq lc$, then
by the above paragraph we may write $a+a^{\prime
}=\tsum\limits_{j=-1}^{r-1}\tsum\limits_{q=1}^{n_{j+1}}\lambda _{j,q}\sigma
_{j,q}$ with $\tsum\limits_{j=-1}^{r-1}\tsum\limits_{q=1}^{n_{j+1}}\lambda
_{j,q}=l$ and\ also $b+b^{\prime }\geq ld+\left\lceil (lc-a+a^{\prime })%
\frac{d}{c}\right\rceil $. Hence, $\left( b+b^{\prime }\right) \func{mod}%
ld\geq \left\lceil (lc-\left( a+a^{\prime }\right) )\frac{d}{c}\right\rceil
=\mu _{a+a^{\prime }}$. Thus $I^{l}$ is Ratliff-Rush by part (2) of Remark (%
\ref{a_i>a/2}).
\end{proof}

\ \ 

\begin{remark}
\ For $c\leq d$, it is known that $I(c,d)=\left\langle x^{_{i}}y^{\mu
_{i}}\mid \mu _{i}=\left\lceil (c-i)\frac{d}{c}\right\rceil ,i=0,\ldots
,c\right\rangle $, the integral closure of the ideal $\left\langle
x^{c},y^{d}\right\rangle $.
\end{remark}

Heinzer et al. \cite{HJLS} asked, Question (1.6) (Q1), whether the minimal
number of generators of a regular ideal is always less than or equal to the
minimal number of generators of its Ratliff-Rush closure. Rossi and Swanson 
\cite{RS}\ answer this question in the following example. In the corollary
below we answers this question by constructing an infinite family of
monomial ideals $I$, with fewer variables, such that the minimal number of
generators of $I$, $\mu (I)$, is arbitrary large and the minimal number of
generators of $\widetilde{I}$ is $5$ (see the figure below for
illustration).:

\begin{example}
(\cite{RS}, Example 3.6) Let $F$ be a field, $n\geq 2$ an integer, $%
x,y,z_{1},\ldots ,z_{n}$ variables over $F$, $R=F[x,y,z_{1},\ldots ,z_{n}]$,
and $I=\left\langle x^{4},x^{3}y,xy^{3},y^{4}\right\rangle
+(x^{2}y^{2})\left\langle z_{1},\ldots ,z_{n}\right\rangle $. Then $%
\widetilde{I}=(x,y)^{4}$, the minimal number of generators of $I$ is $4+n$
and the minimal number of generators of $\widetilde{I}$ is $5$.
\end{example}

\begin{corollary}
Let $c\leq d$ be two integers each of which is divisible by $4$ and $\mu
_{i}=\left\lceil (c-i)\frac{d}{c}\right\rceil $. Let $I=\langle
y^{d},x^{c/4}y^{\mu _{c/4}},x^{3c/4}y^{\mu _{3c/4}},x^{c}\rangle +J$ where $%
J=\langle x^{c/2+c/4-1}y^{\mu _{c/2}},x^{c/2+c/4-2}y^{\mu _{c/2}+2},\ldots ,$
$x^{c/2}y^{\mu _{c/2}+c/4+1}\rangle $. Then $\widetilde{I}=\langle
y^{d},x^{c/4}y^{\mu _{c/4}},x^{c/2}y^{\mu _{c/2}},x^{3c/4}y^{\mu
_{3c/4}},x^{c}\rangle $. In particular, $\mu (I)=c/4+4$ while $\mu (%
\widetilde{I})=5$.
\end{corollary}

\begin{proof}
It is clear that $(c/2,\mu _{c/2})=(c/2,\frac{d}{2})=\left( 2c/4,2\mu _{c/4}%
\func{mod}d\right) $ as $2\mu _{c/4}\func{mod}d\equiv \frac{d}{2}$. Also $%
(c/2,\mu _{c/2})=(3c/4\func{mod}c,2\mu _{3c/4})$. Hence $(c/2,\mu _{c/2})\in
I_{S}\cap I_{T}$. In particular, $I_{S}=I_{T}=\{(0,d),(c/4,\mu
_{c/4}),(c/2,\mu _{c/2}),(3c/4,\mu _{3c/4}),(c,0)\}$ noting $J\subseteq
\left\langle x^{c/2}y^{\mu _{c/2}}\right\rangle $.
\end{proof}

\ \ \ \ \ 

\begin{eqnarray*}
&&\text{ \ \ \ \ \ \ \ \ \ \ \ \ \ \ \ \ \ \ \ \ \ \ }\FRAME{itbpF}{2.041in}{%
2.2105in}{0in}{}{}{Figure}{\special{language "Scientific Word";type
"GRAPHIC";display "USEDEF";valid_file "T";width 2.041in;height
2.2105in;depth 0in;original-width 2.2926in;original-height 2.6264in;cropleft
"0";croptop "1";cropright "1";cropbottom "0";tempfilename
'L873970A.wmf';tempfile-properties "XPR";}} \\
&&\text{\textbf{Figure 5.} {\small The circles (white or black) represent
the generators of} }\widetilde{I}\text{ {\small and }} \\
&&\text{{\small the black points (circles of squares) represent the
generators of the ideal} }I\text{.}
\end{eqnarray*}

\ \ \ \ \

\ \ \ 

{\small Department of Mathematics and Statistics}

{\small Jordan University of Science and Technology}

{\small P O Box 3030, Irbid 22110, Jordan}

{\small Email address: iayyoub@just.edu.jo}

\end{document}